\documentclass[letterpaper, 12pt]{amsart}
\usepackage{graphicx} 
\pdfoutput=1

\usepackage{graphicx,url}%
\usepackage{multirow}%
\usepackage{amsmath,amssymb,amsfonts}%
\usepackage{amsthm}%
\usepackage{mathrsfs}%
\usepackage[title]{appendix}%
\usepackage{xcolor}%
\usepackage{textcomp}%
\usepackage{manyfoot}%
\usepackage{booktabs}%
\usepackage{algorithm}%
\usepackage{algorithmicx}%
\usepackage{algpseudocode}%
\usepackage{listings}%

\newtheorem{theorem}{Theorem}
\newtheorem{proposition}[theorem]{Proposition}%

\newtheorem{example}{Example}%
 \newtheorem{remark}{Remark}%
 \newtheorem{lemma}{Lemma}%
 \newtheorem{problem}{Problem}%
\newtheorem{definition}[theorem]{Definition}

\newcommand{\RR}{\mathbb{R}}
\newcommand{\R}{\mathbb{R}}

\newcommand{\Tn}{\mathcal{U}_m}
\newcommand{\eTn}{\mathcal{U}_m^{\infty}}

\newcommand{\Trop}{\text{Trop}}

\DeclareMathOperator*{\argmin}{arg\,min}

\newcommand{\argmax}{\mathrm{argmax}}
\def\R{{\mathbb R}}

\title{Projected Gradient Descent Method for Tropical Principal Component Analysis over Tree Space}

\author{Ruriko Yoshida}

\date{}
\begin{document}
\maketitle
\begin{abstract}
    {In 2019, Yoshida et al. developed tropical Principal Component Analysis (PCA), that is, an analogue of the classical PCA in the setting of tropical geometry and applied it to visualize a set of gene trees over a space of phylogenetic trees which is an union of lower dimensional polyhedral cones in an Euclidean space with its dimension $m(m-1)/2$ where $m$ is the number of leaves. In this paper, we introduce a projected gradient descent method to estimate the  tropical principal polytope over the space of phylogenetic trees and we apply it to apicomplexa dataset. With computational experiment against Markov Chain Monte Carlo (MCMC) samplers, we show that our projected gradient descent has a lower sum of tropical distances between observations and their projections on an estimated best-fit tropical polytope compared with the MCMC approach proposed by Page et al.~in 2020.}
\end{abstract}

\section{Introduction}
Phylogenomics is a relatively new field that applies tools from phylogenetics to genome data.  One of the tasks in phylogenomics is to analyze {\em gene trees}, which are {\em phylogenetic trees} representing evolutionary histories of genes in the genome.  In this short paper, we focus on an unsupervised learning method to visualize how gene trees are distributed over the {\em space of phylogenetic trees}, that is, the set of all possible phylogenetic trees with a fixed set of labels for all leaves.  

A phylogenetic tree $T$ on the given set of leaves $[m]:= \{1, \ldots , m\}$ is a weighted tree which internal nodes in $T$ are unlabeled and their leaves $X$ are labeled and their branch lengths represent evolutionary clock and mutation rates.  In phylogenetics, a phylogenetic tree in the set of species $[m]$ represents their evolutionary history.  In phylogenomics, we construct a phylogenetic tree from an alignment or sequences for each gene in the given genome. A phylogenetic tree reconstruced from an alignment for a gene is called a gene tree.  Since each gene has different evolutionary history, gene trees do not have to have the same tree topology and their branch lengths might be different.  Thus, it is a challenge to analyze statistically on a set of phylogenetic trees.

When we conduct a statistical analysis on a set of phylogenetic trees, we vectorize trees as vectors in high dimensional vector space.  One way to vectorize a phylogenetic tree is to compute all pairwise distances between two distinct leaves in $[m]$.  This makes a vector in $\mathbb{R}^{\binom{m}{2}}$.   However, any vectors in $\mathbb{R}^{\binom{m}{2}}$ do not corresponding phylogenetic trees on $[m]$. In 1974, Buneman showed in \cite{Buneman} that a vector computed from all possible pairwise distance between leaves in $[m]$ has to satisfy the {\em four point conditions} and for an {\em equidistant tree}, rooted phylogenetic tree whose distance from its root to each leaf in $[m]$ has the same total edge weight (see Definition \ref{df:equidistant}), a vector has to satisfy the {\em three point condition} to realize the phylogenetic tree (Theorem \ref{thm:3pt}).

In 2006, Ardila and Klivans showed that the space of all phylogenetic trees on $[m]$ is an union of $m-2$ dimensional cones in $\mathbb{R}^{\binom{m}{2}}$ and it is not classically convex \cite{AK}.  Therefore, we cannot directly apply a classical statistical method to a set of phylogenetic trees since these methods assume an Euclidean sample space.  

However, Ardila and Klivans also showed that the space of {\em equidistant trees}, rooted phylogenetic trees on $[m]$ defined in Definition \ref{df:equidistant}, is a tropical Grasmaniann so that the space of equidistant trees on $[m]$ is {\em tropically convex} and forms a {\em tropical linear space} with the {\em max-plus algebra} over the {\em tropical projective space}.  Therefore, we can use {\em tropical linear algebra} to conduct statistical analysis on the space of equidistant trees on $[m]$ \cite{tropicalData}. 

In 2019, Yoshida et al.~introduced a {\em tropical principal component analysis} (PCA), which is an analogue of a classical PCA in the view of tropical geometry, to visualize how gene trees are distributed over the space of equidistant trees on $[m]$ using the max-plus algebra \cite{YZZ}.  For $s \leq \binom{m}{2}$, Yoshida et al.~introduced the $(s-1)$-th order {\em tropcial principal polytope} or the {\em best-fit tropical polytope with $s$ vertices}, whose vertices are analogue of the classical first $s$th principal components, and showed that computing a set of vertices of the tropical principal polytope can be formulated as a mixed integer linear programming problem shown in Problem \ref{optimization} \cite{YZZ}. Then Page et al.~later developed a Markov Chain Monte Carlo (MCMC) method to estimate vertices of the tropical principal polytope from a set of gene trees.  

In this short paper, motivated by the recent work of {\em tropical gradient descent} defined in \cite{talbut2024tropicalgradientdescent}, we introduce a projected gradient descent method to compute the set of vertices of the tropical principal polytope from a set of gene trees.  We compute subgradients for finding the optimal solution for the mixed integer programming problem to compute the tropical principal polytope shown in Theorem \ref{tm:subgradient}.  Then we apply our novel method to apicomplexa data from \cite{kuo} and our experiment against the {\tt R} package {\tt TML} \cite{TML} shows that our method has better performance in terms of computational time and cost function.  

This paper is organized as follows:  In Section \ref{BASICS}, we set up basics from tropical geometry.  In Section \ref{sec:treeSpace}, we remind readers notion of metrics and ultrametrics.  Then we discuss that the space of equidistant trees on $[m]$ is isometric to the space of ultrametrics on the finite set $[m]$ using the results by Buneman \cite{Buneman}.  In Section \ref{sec:tropPCA}, we discuss on tropical PCA and the $s$-th order tropical principal polytope for $s \leq e$ where $e:=\binom{m}{2}$.  Section \ref{sec:experiment} shows experimental results on apicomplexa dataset from \cite{kuo}. 

\section{Tropical Basics}\label{BASICS}

In this section, we set up readers with some basics from tropical geometry for our main results. Let ${\bf 1} = (1, \ldots , 1) \in \mathbb{R}^e$. Then, through this paper, we consider the {\em tropical projective torus}, $\R^e/\R\mathbf{1}$ which is isomorphic to $\R^{e-1}$. This means that $\R^e/\R\mathbf{1}$ is equivalent to a hyperplane in $\R^e$. This equivalence implies that for a point $x:=(x_1, \ldots , x_e)\in \mathbb R^e \!/\mathbb R {\bf 1}$, 
\begin{equation*}\label{eq:trop_torus}
    (x_1, \ldots , x_e) = (x_1+c, \ldots , x_e+c)
\end{equation*}
where $c\in\R$. 
See~\cite{MS} and~\cite{joswigBook} for more details.

Throughout this paper, we consider tropically convex sets defined by the max-plus algebra shown in Definition~\ref{def:arith}.  

\begin{definition}[Tropical Arithmetic Operations]\label{def:arith}
    The tropical semiring $(\,\mathbb{R} \cup \{-\infty\},\oplus,\odot)\,$ is defined with the following tropical addition $\oplus$ and multiplication $\odot$:
    $$a \oplus b := \max\{a, b\}, ~~~~ a \odot b := a + b $$
    for any $a, b \in \mathbb{R}\cup\{-\infty\}$. 
\end{definition}
\begin{remark}
    $-\infty$ is the identity element under addition $\oplus$ and $0$ is the identity element under multiplication $\odot$ over $(\,\mathbb{R} \cup \{-\infty\},\oplus,\odot)$.
\end{remark}
\begin{definition}[Tropical Scalar Multiplication and Vector Addition]
    For any $a,b \in \mathbb{R}\cup \{-\infty\}$ and for any $v = (v_1, \ldots ,v_e),\; w= (w_1, \ldots , w_e) \in (\mathbb{R}\cup\{-\infty\})^e$, tropical scalar multiplication and tropical vector addition are defined as:
    $$a \odot v \oplus b \odot w := (\max\{a+v_1,b+w_1\}, \ldots, \max\{a+v_e,b+w_e\}).$$
\end{definition}

\begin{definition}[Generalized Hilbert Projective Metric]
    \label{eq:tropmetric} 
    For any points $v:=(v_1, \ldots , v_d), \, w := (w_1, \ldots , w_e) \in \mathbb R^e \!/\mathbb R {\bf 1}$,  the {\em tropical metric}, $d_{\rm tr}$, between $v$ and $w$ is defined as:
    \begin{equation}\label{eq:tropdist}
    d_{\rm tr}(v,w)  := \max_{i \in \{1, \ldots , e\}} \bigl\{ v_i - w_i \bigr\} - \min_{i \in \{1, \ldots , e\}} \bigl\{ v_i - w_i \bigr\}.
    \end{equation}
\end{definition}

\begin{remark}
    The tropical metric $d_{\rm tr}$ is metric over $\mathbb R^e \!/\mathbb R {\bf 1}$. 
\end{remark} 

\begin{definition}
    A subset $S\subset \RR^e$ is called \textit{tropically convex} if it contains the point $a \odot  x \oplus b \odot y$ for all $x,y
\in S$ and all $a, b \in \RR$. The \textit{tropical convex hull} or \textit{tropical polytope}, ${\rm tconv}(V)$, of a given finite subset $V \subset \RR^e$ is the smallest tropically convex set containing $V \subset \RR^e$.  In addition ${\rm tconv}(V)$ can be written as the set of all tropical linear combinations
$${\rm tconv}(V) = \{a_1 \odot v_1 \oplus a_2 \odot v_2 \oplus \cdots \oplus a_r \odot v_r : v_1,\ldots,v_r \in V \text{ and } a_1,\ldots,a_r \in \RR\}.$$

Any tropically convex subset $S$ of $\RR^e$ is closed under tropical
scalar multiplication, $\RR \odot S \subseteq S$, i.e., if $x
\in S$, then $x+ c \cdot {\bf 1} \in S \text{ for all } c \in
\RR$. Thus, the tropically convex set $S$ is identified as
its quotient in the tropical projective torus
$\RR^e/\RR {\bf 1}$.
\end{definition}

\begin{definition}[Max-tropical Hyperplane~\cite{ETC}]
    \label{def:trop_hyp} 
    A max-tropical hyperplane $H^{\max}_{\omega}$ is the set of points $x\in \mathbb R^e \!/\mathbb R {\bf 1}$ such that 
    \begin{equation}\label{eq:trop_hyp}
    \max_{i \in \{1, \ldots , e\}} \{x_i + \omega_i\}
    \end{equation} 
    is attained at least twice, where $\omega:=(\omega_1, \ldots, \omega_e)\in \mathbb R^e \!/\mathbb R {\bf 1}$.
\end{definition}

\begin{definition}[Min-tropical Hyperplane~\cite{ETC}]
    \label{def:min_trop_hyp} 
    A min-tropical hyperplane $H^{\min}_{\omega}$ is the set of points $x\in \mathbb R^d \!/\mathbb R {\bf 1}$ such that 
    \begin{equation}\label{eq:min_trop_hyp}
    \min_{i \in \{1, \ldots , e\}} \{x_i + \omega_i\}
    \end{equation} 
    is attained at least twice, where $\omega:=(\omega_1, \ldots, \omega_e)\in \mathbb R^e \!/\mathbb R {\bf 1}$.
\end{definition}

\begin{remark}
A min-tropical hyperplane $H^{\min}_{\omega}$ and a max-tropical hyperplane $H^{\max}_{\omega}$ are tropically convex over $\mathbb (R \cup \{-\infty\})^e \!/\mathbb R {\bf 1}$.    
\end{remark}

\begin{definition}[Max-tropical Sectors from Section 5.5 in \cite{joswigBook}]\label{def:sector}
    For $i\in[e]$, the $i{\rm-th}$ {\em open sector} of $H^{\max}_\omega$ is defined as
    \begin{equation}\label{eq:max_open_sector}
        S^{\max,i}_\omega := \{\mathbf{x}\in \mathbb{R}^e/\mathbb{R}\mathbf{1}\,|\, \omega_i + x_i > \omega_j +x_j, \forall j\neq i\}.
    \end{equation}
    \noindent and the $i{\rm-th}$ {\em closed sector} of $H^{\max}_\omega$ is defined as
    \begin{equation}\label{eq:max_closed_sector}
        \overline{S}^{\max,i}_\omega := \{\mathbf{x}\in \mathbb{R}^e/\mathbb{R}\mathbf{1}\,|\, \omega_i + x_i \geq \omega_j +x_j, \forall j\neq i\}.
    \end{equation}
\end{definition}

\begin{definition}[Min-tropical Sectors]\label{def:min_sector}
    For $i\in[d]$, the $i{\rm-th}$ {\em open sector} of $H^{\min}_\omega$ is defined as
    \begin{equation}\label{eq:min_open_sector}
        S^{\min,i}_\omega := \{\mathbf{x}\in \mathbb{R}^d/\mathbb{R}\mathbf{1}\,|\, \omega_i + x_i < \omega_j +x_j, \forall j\neq i\},
    \end{equation}
    \noindent and the $i{\rm-th}$ {\em closed sector} of $H^{\min}_\omega$ is defined as
    \begin{equation}\label{eq:min_closed_sector}
        \overline{S}^{\min,i}_\omega := \{\mathbf{x}\in \mathbb{R}^d/\mathbb{R}\mathbf{1}\,|\, \omega_i + x_i \leq \omega_j +x_j, \forall j\neq i\}.
    \end{equation}
\end{definition}

\section{Space of Phylogenetic Trees}\label{sec:treeSpace}

A phylogenetic tree is a rooted or unrooted tree whose exterior nodes have unique labels, whose interior nodes do not have labels, and whose edges have non-negative weights.  In this paper we focus on on equidistant tree which is a rooted phylogenetic tree such that a total weight on the path from its root to each leaf on the tree has the same total weight.  Let $[m]:= \{1, \ldots , m\}$ be the set of leaf labels on an equidistant tree $T$. 

\begin{definition}\label{df:equidistant}
    An {\em equidistant tree} $T$ on $[m]$ is a rooted phylogenetic tree on $[m]$ such that the total weight from the root to each leaf $i \in [m]$ is equal to a constant $h > 0$ for all $i \in [m]$.  $h$ is called the hight of $T$.
\end{definition}

\begin{example}\label{eg:equid}
Figure \ref{fig:equidistant} shows an equidistant tree with its height $1$ on $[4]$.
    \begin{figure}[h]
        \centering
        \includegraphics[width=\textwidth]{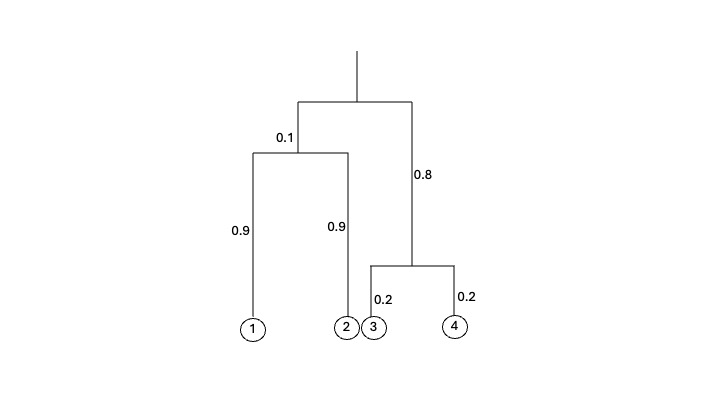}
        \vskip -0.1in
        \caption{An equidistant tree with $[4]$ from Example \ref{eg:equid}.  Its height of the tree is $1$. }
        \label{fig:equidistant}
    \end{figure}
\end{example}
Suppose $D(i, j)$ be the total weight on the unique path from a leaf $i \in [m]$ and a leaf $j \in [m]$ on a phylogenetic tree $T$. Then $D = (D(1, 2), D(1, 3), \ldots , D(m-1, m)) \in \mathbb{R}^e_{\geq 0}$, where $e := \binom{m}{2}$, is a metric, that is, $D$ satisfies 
\begin{eqnarray*}
    D(i, j) \leq D(i, k) + D(k, j)\\
    D(i, j) = D(j, i)\\
    D(i, j) = 0
\end{eqnarray*}
for all $i, j, k \in [m]$.  This metric $D$ is called a {\em tree metric} of a phylogenetic tree $T$.

If a metric $D$ satisfies 
\begin{eqnarray*}
    \max\{D(i, j), D(i, k), D(k, j)\} 
\end{eqnarray*}
achieves at least twice for distinct $i, j, k \in [m]$, then $D$ is called {\em ultrametric}.
Suppose $D(i, j)$ is the total weight of the path from $i, i \in [m]$ from an equidistant tree $T$, then we have the following theorem.

\begin{theorem}[noted in \cite{Buneman}]\label{thm:3pt}
Suppose we have an equidistant tree $T$ 
with a leaf label set $[m]$ and
$D$ as its tree metric.
Then, $D$ is an ultrametric if and only if 
$T$ is an equidistant tree.  In addition, we can reconstruct $T$ from $D$ uniquely.  
\end{theorem}

Using Theorem \ref{thm:3pt}, we consider the {\em space of ultrametrics} on $[m]$ as the {\em space of phylogenetic trees}, which is the set of all possible equidistant trees with the leaf set $[m]$.  Let $\Tn$ be the space of ultrametrics on $[m]$.  

With tropical geometry one can show that $\Tn$ is a tropical subspace over the tropical projective space $(\mathbb{R} \cup \{-\infty\})^e / \RR {\bf 1}$.  
Let $L_m$ denote the subspace of $\mathbb R^e$ defined by the linear equations such that $x_{ij} - x_{ik} + x_{jk}=0$ for $1\leq i < j <k \leq m$. The tropicalization $\Trop(L_m)\subseteq (\mathbb{R} \cup \{-\infty\})^e/\RR {\bf 1}$ is the tropical linear space consisting of points $(u_{12},u_{13},\ldots, u_{m-1,m})$ such that $\max(u_{ij},u_{ik},u_{jk})$ achieves at least twice for distinct $i, j, k \in [m]$.

In addition, it is important to note that the tropical linear space $\Trop(L_m)$ corresponds to the graphic matroid of the complete graph $K_m$.

\begin{theorem}[Theorem 2.18 in \cite{YZZ}]
\label{ultrametrics}
The image of $\mathcal U_m$ in the tropical projective torus $\RR^e/\RR {\bf 1}$ coincides with $\Trop(L_m)$.
\end{theorem}

\subsection{Projection onto Tree Space}

In tropical geometry it is well-known that $\Tn$ is the support of a pointed simplicial fan
of dimension $m-2$ and it has $2^m - m - 2$ rays defined as {\em clade metrics} \cite{AK}.  

\begin{definition}\label{def:clade}
Suppose we have an equidistant phylogenetic tree $T$ with the leave set $[m]$.  A {\em clade} of $T$ with leaves $\sigma \subset [m]$ is an equidistant tree constructed from $T$ by adding all common ancestral interior nodes of any combinations of only leaves $\sigma$ and excluding common ancestors including any leaf from $X - \sigma$   in $T$, and all edges in $T$ connecting to these ancestral interior nodes and leaves  $\sigma$.  
\end{definition}
We note that a clade of an equidistant tree $T$ with leave set $\sigma \subset [m]$ is a subtree of $T$ with the leaves $\sigma$.  Feichtner showed that we can encode each topology of equidistant trees by a {\em nested set}, that is, a set of clades $\{\sigma_1, \ldots, \sigma_C\}$, where $C \in \{1, \ldots , m-2\}$ such that 
\[
\sigma_i \subset \sigma_j, \mbox{ or } \sigma_j \subset \sigma_i, \mbox{ or } \sigma_i \cap \sigma_j = \emptyset
\]
for all $1 \leq  i \leq j \leq C$ and $|\sigma_k| \geq 2$ for all $k = 1, \ldots , C$
\cite{Feichtner2004ComplexesOT}.  
\begin{definition}[Clade Ultrametrics]
    We consider an equidistant tree $T$ on leaves $[m]$. Let $\sigma \subset [m]$ be a proper subset of $[m]$ with at least two elements.  Let $D_{\sigma}:=(D_{\sigma}(1, 2), \ldots, D_{\sigma}(m-1, m)) \in \Tn$ such that 
    \[
    D_{\sigma}(i, j) =\begin{cases}
        0 & \mbox{if } i, j \in \sigma\\
        1 &\mbox{otherwise.}
    \end{cases}
    \]
    Then $D_{\sigma}$ is called a {\em clade ultrametric}.
\end{definition}

We note that Ardila and Klivans showed that a set of clade ultrametrics is a set of generators, i.e., rays, of pointed simplicial fan of dimension $m - 2$ in terms of a classical arithmetic over an Euclidean geometry.  We use an {\em extreme clade ultrametric}, which is an analogue of a clade ultrametric in terms of the max-plus algebra by replacing the identity element of the classical addition with the identity of the tropical addition $\oplus$ (namely, replacing $0$ with $-\infty$) and replacing the identity element of the classical multiplication with the identity of the tropical multiplication $\odot$ (namely, replacing $1$ with $0$).   

\begin{definition}[Extreme Clade Ultrametrics]\label{def:ex_clad_ultra}
    We consider an equidistant tree $T$ on leaves $[m]$. Let $\sigma \subset [m]$ be a proper subset of $[m]$ with at least two elements.  Let $D_{\sigma}:=(D_{\sigma}(1, 2), \ldots, D_{\sigma}(m-1, m)) \in \Tn$ such that 
    \[
    D_{\sigma}(i, j) =\begin{cases}
        -\infty & \mbox{if } i, j \in \sigma\\
        0 &\mbox{otherwise.}
    \end{cases}
    \]
    Then $D_{\sigma}$ is called an {\em extreme clade ultrametric}.
\end{definition}

\begin{remark}
    In polyhedral geometry, a polyhedral cone generated by a set of rays $V=\{v^1, \ldots , v^k\} \subset \RR^{e-1}$ is defined as 
    \[
    C(V) = \left\{x \in \RR^{e-1} \Bigg| x = \sum_{i=1}^k \alpha_i v^i, \, \alpha_i \geq 0 \mbox{ for all }i = 1, \ldots , k\right\}.
    \]
    We replace the classical addition with $\oplus$ and the classical multiplication with $\odot$ for $V=\{v^1, \ldots , v^k\} \subset \RR^e/\RR{\bf 1} \cong \RR^{e-1}$, then we have
    \[
    Trop(C(V)) = \left\{x \in \RR^e/\RR{\bf 1} \Bigg| x = \bigoplus_{i=1}^k \alpha_i \odot v^i, \, \alpha_i \in \RR \mbox{ for all }i = 1, \ldots , k\right\}
    \]
    which is a tropical polytope defined by $V$.
\end{remark}

\begin{proposition}[\cite{YCT}]\label{pro:generatingset}
    The set of all extreme clade ultrametrics, $\eTn$, is a generating set of $\Tn$ in terms of the max-plus algebra.
\end{proposition}
Proposition \ref{pro:generatingset} is a tropical geometric analogue of the simplicial complex as a result by Ardila and Klivans in \cite{AK} by replacing a classical addition with $\oplus$ and a classical multiplication with $\odot$.  

\begin{definition}[Projection Map \cite{YCT}]
    \label{defn:tropicalProjectionMap}
    The tropical projection map to ultrametric tree space $\pi_{\mathcal{U}_m}: (\mathbb{R} \cup \{-\infty\})^e/\RR{\bf 1} \to \Tn$  is given by
    \[
        \pi_{\mathcal{U}_m}(x) = \bigoplus_{v \in \mathcal{U}_m^{\infty}} \, \lambda_v \odot v, \qquad \lambda_v = \min_{j : \, v_j = 0}\{ x_j \}
    \]
for $x := (x_1, \ldots , x_e) \in \RR^e/\RR{\bf 1}$.
\end{definition}
\begin{proposition}[\cite{YCT}]
    For all $x \in \RR^e/\RR{\bf 1}$, we have
    \[
    d_{\rm tr}(x, x') \leq d_{\rm tr}(x, y)
    \]
    for all $y \in \Tn$ and where $x' = \pi_{\mathcal{U}_m}(x)$ defined by Definition \ref{defn:tropicalProjectionMap}.
\end{proposition}

The following proposition is a key for the {\em projected gradient methods} which we proposed in Section \ref{sec:tropPCA}.
\begin{proposition}[\cite{YCT}]\label{prop:nonExp}
    The projection map $\pi_{\mathcal{U}_m}(x)$ is non-expansive in terms of $d_{\rm tr}$, i.e., 
    \[
    d_{\rm tr}(\pi_{\mathcal{U}_m}(u), \pi_{\mathcal{U}_m}(v)) \leq d_{\rm tr}(u, v)
    \]
    for all $u, v \in \RR^e/\RR {\bf 1}$.
\end{proposition}

\begin{remark}
    The projection map $\pi_{\mathcal{U}_m}$ is equivalent to the single linkage hierarchical clustering method \cite{YCT}.  Therefore, in the computational experiment described in Section \ref{sec:experiment}, we use a single linkage hierarchical clustering method to projecting a subgradient.  
\end{remark}

\section{Tropical Principal Component Analysis}\label{sec:tropPCA}

Yoshida et al.~defined a notion of {\em tropical principal component analysis (PCA)}, an analysis using the best fit tropical hyperplane or tropical polytope \cite{YZZ}.  Especially they applied tropical PCA with tropical polytopes to a sample of ultrametrics over the space of ultrametrics.  
In this section we consider $\Tn \subset \mathbb{R}^e/\mathbb{R}{\bf 1}$ where $m$ is the number of leaves and $e = \binom{m}{2}$.  Suppose we have a sample $\mathcal{S}:=\{u_1, \ldots u_n\} \subset \Tn$.
\begin{definition}[Definition 3.1 in \cite{10.1093/bioinformatics/btaa564}]
    The $s$-th order tropical principal polytope $\mathcal{P}$ whose minimizes 
    \[
    \sum_{i=1}^n d_{\rm tr}(u_i, w_i)
    \]
    where $w_i$ is the projection onto $\mathcal{P}$, that is 
    \[
    d_{\rm tr}(u_i, w_i) \leq d_{\rm tr}(u_i, x)
    \]
    for all $x \in \mathcal{P}$ for 
    $\mathcal{S}:\{u_1, \ldots u_n\}$ is called the {\rm $(s-1)$th-order tropical principal component polytope} of the sample $\mathcal{S}$.  $s$ many vertices of the tropical principal component polytope $\mathcal{P}$ is called {\em $(s-1)$th-order tropical principal components} or we call the {\rm best-fit tropical polytope with $s$ vertices}.
\end{definition}

\begin{remark}
    The $0$-th tropical principal polytope is a {\em tropical Fermat-Weber point} of a sample $\mathcal{S}$ with respect to $d_{\rm tr}$.  A tropical Fermat Weber point of $\mathcal{S}$ with respect to $d_{\rm tr}$ is defined as 
    \[
    x^* := \argmin_{x \in \mathbb{R}^e/\mathbb{R}{\bf 1}} \sum_{i=1}^n d_{\rm tr}(x, u_i).
    \]
    A tropical Fermat Weber point of $\mathcal{S}$ with respect to $d_{\rm tr}$ is not unique and a set of tropical Fermat Weber points forms a classical polytope \cite{Lin2016TropicalFP} and tropical polytope \cite{BSYM}. 
\end{remark}

In this paper, we focus on the $(s-1)$-th order principal components over $\Tn \subset \RR^e/\RR{\bf 1}$ for $s > 1$.  Our problem can be written as follows:
\begin{problem}\label{optimization}
We seek a solution for the following optimization problem:
\[
\min_{D^{(1)}, \ldots , D^{(s)} \in \Tn} \sum_{i = 1}^n  d_{\rm
  tr}(u_i, w_i)
\]
where
\begin{equation}\label{const1:convex}
w_i= \lambda_1^i \odot  D^{(1)} \,\oplus \,
\ldots \,\oplus \,
\lambda_s^i \odot  D^{(s)}  ,
\quad {\rm where} \,\, \lambda_k^i = {\rm min}(u_i-D^{(k)}) ,
\end{equation}
and 
\begin{equation}\label{const2:convex}
d_{\rm
  tr}(u_i, w_i) = \max\{|u_i(k) - w_i(k) - u_i(l) + w_i(l)|: 1 \leq
k < l \leq e\}
\end{equation}
with 
\begin{equation}\label{const3:convex}
u_i = (u_i(1), \ldots , u_i(e)) \text{ and } w_i = (w_i(1), \ldots , w_i(e)).
\end{equation}
\end{problem}

\begin{remark}[Proposition 4.2 in \cite{YZZ}]
    Problem \ref{optimization} can be formulated as a mixed integer programming problem.  
\end{remark}

In this section we consider subgradients of Problem \ref{optimization}. Here we are interested in computing
\[
\frac{\partial d_{\rm tr}(u_i, w_i)}{\partial D^{(k)}}.
\]
First, we notice that
\begin{eqnarray*}
    \frac{\partial d_{\rm tr}(u_i, w_i)}{\partial D^{(k)}} &=& \frac{\partial d_{\rm tr}(u_i, w_i)}{\partial w_i} \frac{\partial w_i}{\partial D^{(k)}}
\end{eqnarray*}
by the product rule.  

Let 
\[
\delta_{ij} = \begin{cases}
    1 & \mbox{if } i = j\\
    0 & \mbox{otherwise.}
\end{cases}
\]
be the Kronecker's Delta.  Then we have the following lemma:
\begin{lemma}[Lemma 10 in \cite{BSYM}]
For any two points $x,p\in \mathbb R^e \!/\mathbb R {\bf 1}$, the gradient at $x$ of the tropical distance between $x$ and $p$ is given by
    \begin{equation}\label{eq:partial1}
        \frac{\partial d_{\rm tr}(p, x)}{\partial x(l')} = \left(\delta_{l't} - \delta_{l't'}\mid x\in {S}^{\max, t}_{-p} \cap {S}^{\min, t'}_{-p}\right).
    \end{equation}
if there are no ties in $(x-p)$, implying that the min- and max-sectors are uniquely identifiable, that is, the point $x$ is inside of open sectors and not on the boundary of $H_{-p}$.
\end{lemma}

Also we notice that
\begin{eqnarray*}\label{eq:proj}
    w_i(l) 
    &=& \max\left[\left(\min_j(u_i(j) - D^{(1)}(j)){\bf 1} + D^{(1)}\right)(l), \ldots , \left(\min_j(u_i(j) - D^{(s)}(j)){\bf 1} + D^{(s)} \right)(l)\right],
\end{eqnarray*}
for $l = 1, \ldots , e$.
\begin{lemma}
\begin{equation}\label{eq:partial2}
      \frac{\partial w_i(l')}{\partial D^{(k)}(l)} = \begin{cases}
        -1 & \mbox{if } \argmax_{k'}\left\{\left(D^{(k')} + \min_j(u_i(j) - D^{(k')}(j)){\bf 1}\right)(l')\right\} = k,\\
         &  \min_j(u_i(j) - D^{(k)}(j)) = l,\mbox{ and } l' \not = l,\\
        1 & \mbox{if } \argmax_{k'}\left\{\left(D^{(k')} + \min_j(u_i(j) - D^{(k')}(j)){\bf 1}\right)(l')\right\} = k,\\
         & \min_j(u_i(j) - D^{(k)}(j)) \not = l,\mbox{ and } l' = l,\\
         0 &\mbox{otherwise,}
    \end{cases}  
\end{equation}
    for $i = 1, \ldots , n$, $k = 1, \ldots , s$, and for $l = 1, \ldots , e$.
\end{lemma}

\begin{proof}
    Direct computation from the equation in \eqref{eq:proj}.  
\end{proof}

\begin{lemma}\label{lm:subgradient}
    Subgradients of Problem \ref{optimization} over $\RR^e/\RR {\bf 1}$ is 
    \[
\sum_{i=1}^n \frac{\partial d_{\rm tr}(u_i, w_i)}{\partial D^{(k)}(l)} = \sum_{i=1}^n \frac{\partial d_{\rm tr}(u_i, w_i)}{\partial w_i(l')} \frac{\partial w_i(l')}{\partial D^{(k)}(l)}
    \]
    which can be obtained by equations in \eqref{eq:partial1} and \eqref{eq:partial2}.
\end{lemma}

\begin{theorem}\label{tm:subgradient}
    Subgradient of Problem \ref{optimization} over $\Tn$ is 
    \[
\pi_{\eTn}\left(\sum_{i=1}^n \frac{\partial d_{\rm tr}(u_i, w_i)}{\partial D^{(k)}(l)}\right) 
    \]
    where $\sum_{i=1}^n \frac{\partial d_{\rm tr}(u_i, w_i)}{\partial D^{(k)}(l)}$ is obtained in Lemma \ref{lm:subgradient}.
\end{theorem}
\begin{proof}
    By Proposition \ref{prop:nonExp}, we know that $\pi_{\eTn}$ is non-expanding in terms of $d_{\rm tr}$.  Therefore we have
    \[
    d_{\rm}\left(0, \pi_{\eTn}\left(\sum_{i=1}^n \frac{\partial d_{\rm tr}(u_i, w_i)}{\partial D^{(k)}(l)}(x) \right) \right) \leq d_{\rm tr}\left(0, \sum_{i=1}^n \frac{\partial d_{\rm tr}(u_i, w_i)}{\partial D^{(k)}(l)}(x)\right).
    \]
    Therefore, \[
d_{\rm}\left(0, \pi_{\eTn}\left(\sum_{i=1}^n \frac{\partial d_{\rm tr}(u_i, w_i)}{\partial D^{(k)}(l)}(x)\right) \right) = 0
    \]
    when $x$ is at a critical point.

    Suppose $x^* \in \Tn$ is an optimal solution for the Problem \ref{optimization}.  Then let 
    \[
    x^{t+1} := x^t - \alpha_t \sum_{i=1}^n \frac{\partial d_{\rm tr}(u_i, w_i)}{\partial D^{(k)}(l)}(x^t).
    \]
    Since 
    \[
    \sum_{i=1}^n \frac{\partial d_{\rm tr}(u_i, w_i)}{\partial D^{(k)}(l)}
    \]
    is subgradient by Lemma \ref{lm:subgradient}, we have
    \[
    \begin{array}{lcr}
         &&\sum_{i = 1}^n  d_{\rm tr}(u_i, \pi_{tconv(D^{(1)}, \ldots , D^{(k-1)}, x^{t+1}, D^{(k+1)}, \ldots , D^{(s)})}(u_i))\\ 
         &\leq& \sum_{i = 1}^n  d_{\rm tr}(u_i, \pi_{tconv(D^{(1)}, \ldots , D^{(k-1)}, x^{t}, D^{(k+1)}, \ldots , D^{(s)})}(u_i)),
    \end{array}
    \]
    for $k = 1, \ldots , s$.
    Since Proposition \ref{prop:nonExp}, we have
    \[
    \begin{array}{lcr}
         &&\sum_{i = 1}^n  d_{\rm tr}\left(\pi_{\eTn}(u_i), \pi_{\eTn}\left(\pi_{tconv(D^{(1)}, \ldots , D^{(k-1)}, x^{t+1}, D^{(k+1)}, \ldots , D^{(s)})}(u_i)\right)\right)\\ 
         &=& \sum_{i = 1}^n  d_{\rm tr}\left(u_i, \pi_{\eTn}\left(\pi_{tconv(D^{(1)}, \ldots , D^{(k-1)}, x^{t+1}, D^{(k+1)}, \ldots , D^{(s)})}(u_i)\right)\right)\\
         &\leq & \sum_{i = 1}^n  d_{\rm tr}(u_i, \pi_{tconv(D^{(1)}, \ldots , D^{(k-1)}, x^{t+1}, D^{(k+1)}, \ldots , D^{(s)})}(u_i))\\ 
         &\leq& \sum_{i = 1}^n  d_{\rm tr}(u_i, \pi_{tconv(D^{(1)}, \ldots , D^{(k-1)}, x^{t}, D^{(k+1)}, \ldots , D^{(s)})}(u_i)),
    \end{array}
    \]
    for $k = 1, \ldots , s$.
\end{proof}

\vspace{6pt} 





\section{Computational Experiments}\label{sec:experiment}

\subsection{Simulated Dataset}

In this experiments, we generate gene trees from the multispecies coalescent model (MCM) with a given species tree via the software {\tt Mesquite} \cite{mesquite}.  We fixed
the effective population size $N_e = 100,000$ and varied
$R = \frac{SD}{N_e}$
where $SD$ is the species depth which is the number of generations from the common ancestor (the root of the tree) to its leaves. 

In this experiment, we simulate $1000$ trees for each $R = 0.25, 0.5, 1.0, 2.0, 5.0, 10.0$ and we set iteration to be 100 with a MCMC method as well as our projected gradient method.  We recompute a MCMC technique and our projected gradient method 10 times for each $R$ so that we can capture variation of each sum of tropical distance between observations and their projections on an estimated tropical polytope (sum of "magnitude of errors" (SE) in terms of $d_{\rm tr}$), i.e., an estimated optimal value of the linear programming problem in Problem \ref{optimization}.   The results are shown in Figure \ref{fig:SSE}.  For the scheduling of the learning rate we start $\lambda = 0.001$.  Then each iteration we multiply the learning rate by $0.999$.  Note that this learning rate is not optimized.  Therefore, we could investigate what is the optimal learning rates.  We set the stopping criteria if the iteration reaches to the maximum iteration which we preset. Each iteration should take $O(m^2n)$, where $n$ is the sample size.

\begin{figure}
    \centering
    \includegraphics[width=\textwidth]{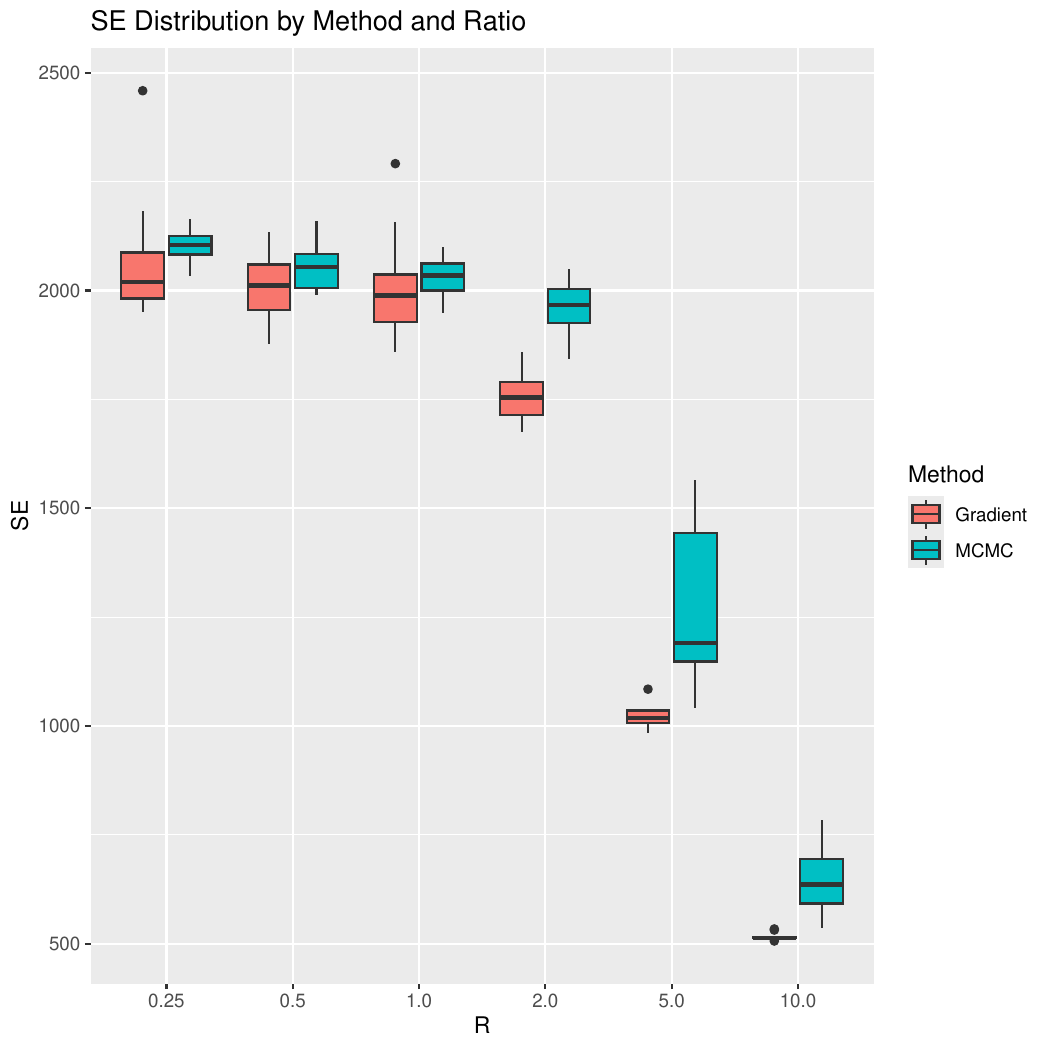}
    \caption{Side-by-Side Boxplots for SE for each method and each different ratio. We repeat computation 10 times for each ration and each method.}
    \label{fig:SSE}
\end{figure}

\subsection{Empirical Dataset}

We apply our projected gradient method for estimating the best-fit tropical polytope for the following empirical dataset from  268  orthologous
sequences  with  eight  species  of  protozoa  presented  in  \cite{kuo}.
This data set has  gene  trees  reconstructed  from  the  following
sequences: {\it Babesia  bovis} (Bb), {\it Cryptosporidium
  parvum} (Cp), {\it Eimeria tenella} (Et) [15], {\it
  Plasmodium falciparum} (Pf) [11], {\it Plasmodium  vivax} (Pv),
{\it Theileria  annulata} (Ta),  and {\it Toxoplasma
  gondii} (Tg).  An outgroup is a free-living ciliate, {\it
  Tetrahymena  thermophila} (Tt).
  
\begin{figure}
    \centering
    \includegraphics[width=0.4\textwidth]{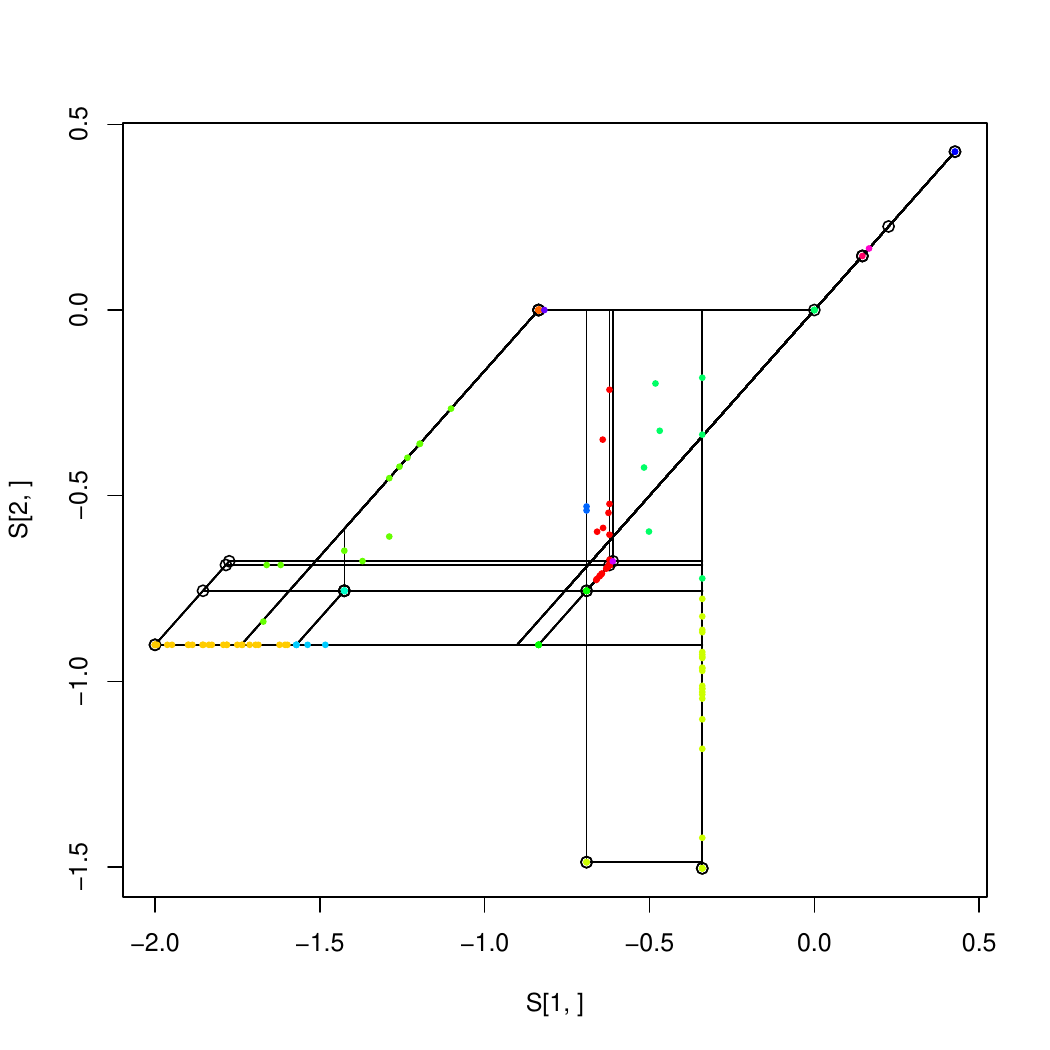}
    \includegraphics[width=0.5\textwidth]{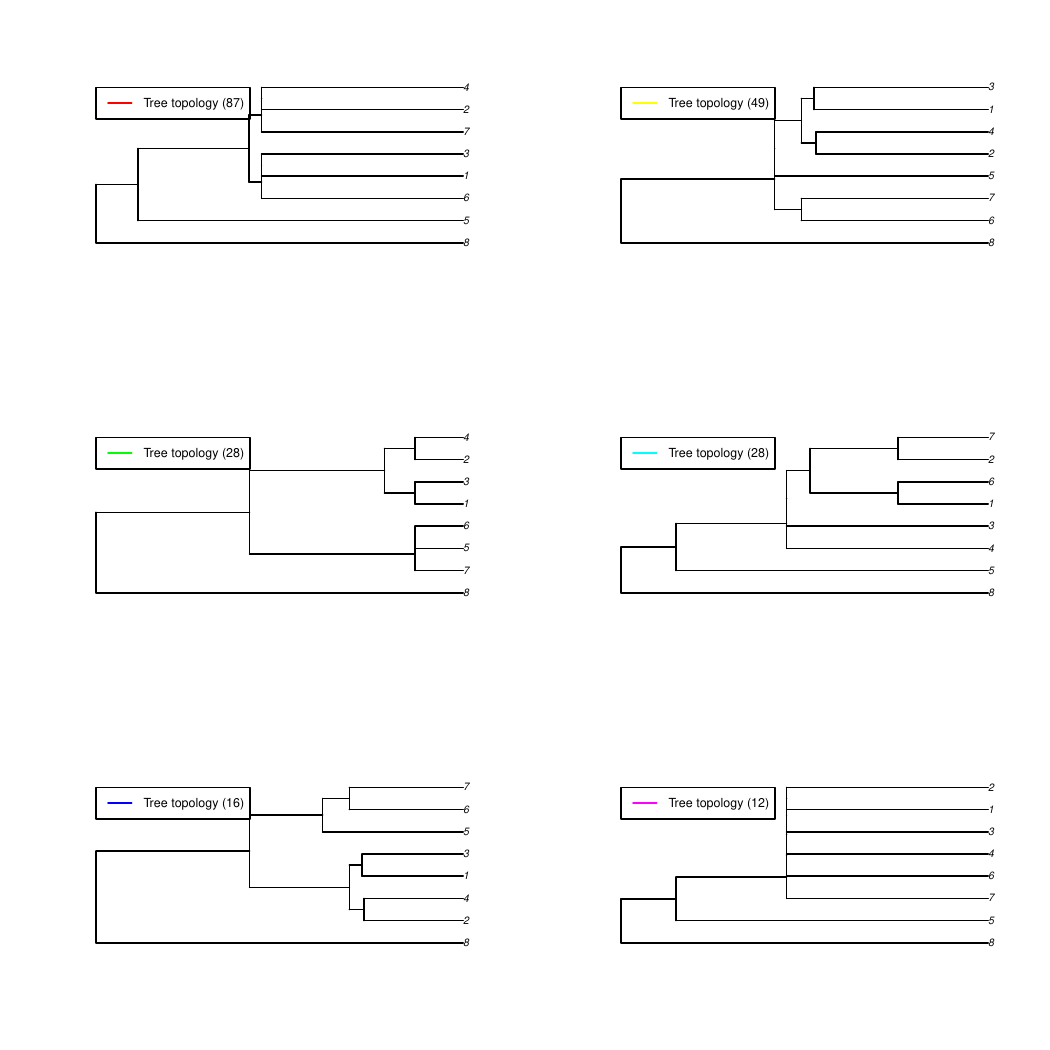}
    \caption{Left: Estimated second order tropical principal polytope. Right: Each color represents a tree topology. The number inside of each branket is the frequency of the tree topology.  1 presents ``Pv'', 2 represents ``Pf'', 3 represents ``Tg'', 4 represents ``Et'', 5 represents ``Cp'', 6 represents ``Ta'', 7 represents ``Bb'' and 8 represents ``Tt''.}
    \label{fig:PCA}
\end{figure}

In order to run this experiment, we use a Mac Pro laptop with Apple M4 Max and 128 GB memory. We implement our projected gradient method using {\tt R}.

We set the maximum iteration to be 100 for our projected gradient method and for a MCMC method, we set the maximum iterations to be 1000.  For the scheduling of the learning rate we start $\lambda = 0.01$.  Then each iteration we multiply the learning rate by $0.999$.  Note that this learning rate and scheduling are not optimized. 

Estimating the tropical principal polytope via the projected gradient descent on this dataset takes $6.82$ seconds and an estimated optimal value over the optimization problem in Problem \ref{optimization} is $360.8589$ while an estimated optimal value via the Markov Chain Monte Carlo (MCMC) sampler from the {\tt TML} package \cite{TML} is $397.6459$ with its computational time $54.70$ seconds with $1000$ iterations.

\section{Conclusion}

In this short paper, we introduce a novel method to approximate the best-fit tropical polytope to explain a sample of gene trees.  We show that this gradient method reduces the objective function with appropriate learning rate and it has a lower sum of tropical distances between observations and their projections on an estimated best-fit tropical polytope compared with the MCMC approach proposed by Page et al.~\cite{10.1093/bioinformatics/btaa564} from the computational experiment.   In an experiment, we implement a learning rate that decreases with each iteration, but it is not clear how we schedule learning rates for this problem.  

We implement our method in {\tt R} and the source code is available at \url{http://polytopes.net/Tropical_Gradient2.zip}.  


\bibliographystyle{plain} 
\bibliography{refs}

\end{document}